\documentclass{article}
\usepackage[utf8]{inputenc}

\usepackage{amssymb,amsmath,amscd,amsthm}
\usepackage{wasysym}
\usepackage{graphicx}
\usepackage{wrapfig}

\usepackage{thm-restate}
\newtheorem{theorem}{Theorem}[section]
\newtheorem{corollary}[theorem]{Corollary}

\newtheorem{lemma}[theorem]{Lemma}
\newtheorem{proposition}[theorem]{Proposition}

\usepackage{geometry}
\newcommand{\Bound}{\partial\Omega}
\newcommand{\D}{\text{D}}
\newcommand{\quat}{\mathbb{H}}

\title{Uniqueness of the inverse conductivity problem once-differentiable complex conductivities in three dimensions}
\author{Ivan Pombo}
\date{June 2022}

\begin{document}

\maketitle
\abstract{We prove uniqueness of the inverse conductivity problem in three dimensions for complex conductivities in $W^{1,\infty}$.

We apply quaternionic analysis to transform the inverse problem into an inverse Dirac scattering problem, as established in two dimensions by Brown and Uhlmann. 

This is a novel methodology that allows to extend the uniqueness result from once-differentiable real conductivities to complex ones.}

\section{Introduction}

Let $\gamma \in W^{1,\infty}(\Omega)$ be the complex-valued conductivity defined in a bounded Lipschitz domain $\Omega\subset \mathbb{R}^3$ and given by $\gamma = \sigma + i\omega \epsilon$ where $\sigma$ is the electrical conductivity and satisfies $\sigma\geq c>0$, $\epsilon$ is the electrical permittivity and $\omega$ is the current frequency.

Given a boundary value $f\in H^{1/2}(\Bound)$ we can determine the respective electrical potential $u\in H^1(\Omega)$ by uniquely solving
\begin{align}\label{Conductivity_eq1}
    \begin{cases}
        \nabla\cdot(\gamma\nabla u)=0, \text{ in } \Omega,\\
        \left. u \right|_{\Bound}  = f.
    \end{cases}
\end{align}

This is the so-called conductivity equation which describes the behavior of the electrical potential, $u$, in a conductive body when a voltage potential is applied on the boundary, $f$.

In 1980, A.P. Calderón, \cite{Calderon} introduced the problem of whether one can recover uniquely a conductivity $\sigma \in L^{\infty}(\Omega)$ from the boundary measurements, i.e., 
from the Dirichlet-to-Neumann map
\begin{align}\label{Dirichlet_to_Neumann_map}
    \Lambda_{\sigma}:\, H^{1/2}(\Bound) &\rightarrow H^{-1/2}(\Bound),\\
                   \nonumber    f &\mapsto \sigma \left.\frac{\partial u}{\partial\nu}\right|_{\Bound}
\end{align}
which connects the voltage and electrical current at the boundary. The normal derivative exists as an element of $H^{-1/2}(\Bound)$ by
\begin{align}\label{Dirichlet_to_Neumann_map_weak}
    \langle \Lambda_\sigma f, g\rangle = \int_{\Omega} \sigma\nabla u\cdot\nabla v\,dx
\end{align}
where $v\in H^1(\Omega)$ with $\left.v\right|_{\Bound} = g$ and $u$ solves (\ref{Conductivity_eq1}).

In the same paper, Calderón was able to prove that the linearized problem at constant real conductivities has a unique solution. Thereafter, many authors extended is work into global uniqueness results. Sylvester and Uhlmann \cite{Sylvester_Uhlmann} used ideas of scattering theory, namely the exponential growing solutions of Faddeev \cite{Faddeev} to obtain global uniqueness in dimensions $n\geq 3$ for smooth conductivities. Using this foundations the uniqueness for lesser regular conductivities was further generalized for dimensions $n\geq 3$ in the works of (\cite{Alessandrini}, \cite{Brown}, \cite{Brown_Torres},  \cite{Caro_Rodgers}, \cite{Chanillo}, \cite{Haberman_Tataru},  \cite{Nachman3D}, \cite{Nach_Syl_Uhl}, \cite{PPU}). Currently, the best know result is due to Haberman \cite{Haberman} for conductivities $\gamma\in W^{1, 3}(\Omega)$. The reconstruction procedure for $n\geq 3$ was obtained in both \cite{Nachman3D} and \cite{Novikov} independently.  As far as we are aware, there seems to be no literature concerning reconstruction for conductivities with less than two derivatives. 

In two dimensions the problem seems to be of a different nature and tools of complex analysis were used to establish uniqueness. Nachman \cite{Nachman2D} obtained uniqueness and a reconstruction method for conductivities with two derivatives. The uniqueness result was soon extend for once-differentiable conductivities in \cite{Brown_Uhlmann} and a corresponding reconstruction method was obtained in \cite{Knudsen_Tamasan}. In 2006, Astala and Päivärinta \cite{Astala_Paivarinta} gave a positive answer Calderón's problem for $\sigma\in L^{\infty}(\Omega),\, \sigma\geq c>0$, by providing the uniqueness proof through the reconstruction process.

All of this definitions can be extended to the complex-conductivity case with the assumption $\text{Re } \gamma\geq c>0$, in particular, we can define the Dirichlet-to-Neumann as above $\Lambda_{\gamma}$.

In this scenario, the first works was done in two-dimensions by Francini \cite{Francini}, by extending the work of Brown and Uhlmann \cite{Brown_Uhlmann} in two-dimensions proving uniqueness for small frequencies $\omega$ and $\gamma \in W^{2,\infty}$. Afterwards, Bukgheim influential paper \cite{Bukgheim} proved the general result in two-dimensions for complex-conductivities in $W^{2,\infty}$. He reduced the (\ref{Conductivity_eq1}) to a Schrödinger equation and shows uniqueness through the stationary phase method (based on is work many extensions followed in two-dimensions \cite{Astala_Faraco_Rodgers}, \cite{Imanuvilov_Yamamoto}, \cite{Novikov_Santacesaria}). Recently, by mixing techniques of \cite{Brown_Uhlmann} and \cite{Bukgheim}, Lakshtanov, Tejero and Vainberg obtained \cite{Lakshtanov_Tejero_Vainberg} uniqueness for Lipschitz complex-conductivities in $\mathbb{R}^2$. In \cite{Pombo}, the author followed up their work to show that it is possible to reconstruct complex-conductivity with a jump at least in a certain set of points.

In three dimensions, the uniqueness results presented in \cite{Sylvester_Uhlmann} and \cite{Nach_Syl_Uhl} hold for twice-differentiable complex-conductivities in $W^{2,\infty}$, but there was no reconstruction process presented. Nachman's reconstruction method in three dimensions \cite{Nachman3D} was used in \cite{Hamilton_3D} to reconstruct complex conductivities from boundary measurements. Even though the Nachman's proof was presented only for real conductivities, the paper \cite{Pombo_3D} structures the proof in order to show the result holds for complex-conductivities. As far as we aware, the works with lower regularity require real-conductivities. 

In this paper our interest resides in Calderón's problem for once-differentiable complex-conductivities in three-dimensions. The aim is to prove the following theorem:

\begin{theorem}
Let $\Omega\subset\mathbb{R}^3$ a bounded Lipschitz domain, $\gamma_i\in W^{1,\infty}(\Omega), \, i=1,\, 2$ be two complex-valued conductivities with $\textnormal{Re }\gamma_i\geq c>0.$ 
$$\text{If } \Lambda_{\gamma_1} = \Lambda_{\gamma_2}, \text{ then } \gamma_1=\gamma_2.$$
\end{theorem}

Our work basis itself on a transformation of (\ref{Conductivity_eq1}) into a Dirac system of equation in three-dimensions with the help of quaternions. In this scenario, we obtain a potential $q$ that we want to determine from boundary data. The main ideas follow the work of Brown and Uhlmann \cite{Brown_Uhlmann} for real conductivities in two-dimensions and Lakshtanov, Vainberg and Tejero \cite{Lakshtanov_Tejero_Vainberg}, as well as the authors work \cite{Pombo}, for complex-conductivities. 

In this paper, we provide a novel reconstruction of the bounded potential $q$ from the boundary data, but we are yet to be able to establish a relation between this boundary data and the Dirichlet-to-Neumann map. This is essentially to answer Calderón problem for Lipschitz complex conductivities, but the lack of a well-suited Poincaré lemma that fits the quaternion structure does not allow such a simple work as in 2D.

\section{Minimalistic lesson of Quaternions}

We present the basis of the quaternionic framework we will use for our work. Let $\mathbb{R}_{(2)}$ be the real universal Clifford Algebra over $\mathbb{R}^2$. By definition, it is generated as an algebra over $\mathbb{R}$ by the elements $\{e_0, e_1, e_2\}$, where $e_1, e_2$ is a basis of $\mathbb{R}^2$ with $e_ie_j + e_je_i = -2\delta_{ij}$, for $i,\, j=1,\, 2$ and $e_0=1$ is the identity and commutes with the basis elements. This algebra has dimension 4 and is identified with the algebra of the quaternions, $\quat$. As such it holds $e_3=e_1e_2$. 
In the following, we refer to this algebra as the quaternions. An element of the quaternions can be written as:
\begin{align}
x = x_0 + x_1e_1 + x_2e_2 + x_3e_3,
\end{align}
where $x_j,\, j=0, ...,3$ are real. We define the quaternionic conjugate $\bar{x}$ of an element $x$ as
\begin{align}
    \bar{x} = x_0 - x_1e_1 - x_2e_2 - x_3e_3.
\end{align}

Let $x,\, y \in \quat$, we write $xy$ for the resulting quaternionic product. The product $\bar{x}y$ defines a Clifford valued inner product on $\quat$. Further, we have $\overline{xy}=\bar{y}\bar{x}$ and the conjugate of the conjugate of quaternion is that same quaternion. Let $x\in\quat$ then $\text{Sc}\, x = x_0$ denotes the scalar of $x$ and $\text{Vec}\,x = x - \text{Sc}\,x$. The scalar of a Clifford inner product $\text{Sc}(\bar{x}y)$ is the usual inner product in $\mathbb{R}^4$ for $x,\, y$ identified as vectors.

With this inner product $\quat$ is an Hilbert space and the resulting norm is the usual Euclidean norm. 

In order to introduce some of the concepts we also extend the real quaternions to complex quaternions $\mathbb{C}_2 = \mathbb{C}\otimes\quat$. Here, we use the same generators $(e_1, e_2)$ as above, with the same multiplication rules, however, the coefficients of the quaternion can be complex-valued. That is, $\lambda \in \mathbb{C}\otimes\quat$ may be written as
\begin{align}
    \lambda = \lambda_0 + \lambda_1 e_1 + \lambda_2e_2 + \lambda_3e_3, \quad \lambda_j \in \mathbb{C}, \, j=0,...,3
\end{align}
or still as
\begin{align}
    \lambda = x + iy,\quad x,\,y \in \quat.
\end{align}

Due to the complexification we can still take another conjugation, to which we define has Hermitian conjugation and denote it by $\overline{\cdot}^\dag$. Explicitly, for $\lambda \in \mathbb{C}\otimes\quat$ one has
\begin{align}
 \bar{\lambda}^\dag = \lambda_0^c - \lambda_1^c e_1 - \lambda_2^c e_2 - \lambda_3^c e_3, 
\end{align}
where $\cdot^c$ denotes complex conjugation, or
\begin{align}
\bar{\lambda}^\dag = \bar{x} - i\bar{y}.
\end{align}

Similarly, one can introduce an associated inner product and norm in $\mathbb{C}\otimes\quat$ by means of this conjugation:
\begin{align}
    \langle\lambda, \mu\rangle = \text{Sc}\left(\bar{\lambda}^{\dag}\mu\right); \quad |\lambda|_{\mathbb{C}_2} = \sqrt{\text{Sc}\left(\bar{\lambda}^{\dag}\lambda\right)}.
\end{align}

For ease of notation, we also define for $\lambda\in \mathbb{C}_2$ the complex conjugation as
\begin{align}
    \bar{\lambda}^c = \lambda_0^c + \lambda_1^c e_1 + \lambda_2^c e_2 + \lambda_3^c e_3.
\end{align}

Now, we can also introduce Quaternion-valued functions $f: \mathbb{R}^3\rightarrow \mathbb{C}_2$ written as $f = f_0 + f_1e_1 + f_2e_2 + f_3e_3,$ where $f_j:\mathbb{R}^3\rightarrow \mathbb{C}$. 

The Banach spaces $L^p,\, W^{n, p}$ of $C_2$-valued functions are defines by requiring that each component is in such space. On $L^2(\Omega)$ we introduce the $C_2$-valued inner product
\begin{align}
    \langle f, g\rangle = \int_{\Omega} \bar{f}^\dag(x)g(x)\, dx.
\end{align}

Analogously to the Wirtinger derivatives in complex analysis, we have the Cauchy-Riemann operators under $(x_0, \,x_1, \,x_2)$  coordinates of $\mathbb{R}^3$ defined as
\begin{align}
    D = \partial_{0} + e_1\partial_{1}+e_2\partial_2,
\end{align}
where $\partial_j$ is the derivative with respect to the $x_j,\, j=0,1,2$ variable; and
\begin{align}
    \bar{D} = \partial_{0} - e_1\partial_{1}-e_2\partial_2.
\end{align}
The vector part of the Cauchy-Riemann operator is designated as Dirac operator.
It holds that $D\bar{D}=\Delta$ where $\Delta$ is the Laplacian.

We designate any function $f$ fulfilling $D f=0$ as a monogenic function, analogous to the holomorphic functions in complex analysis.

\subsection{A bit of Operator Theory}

Let $\Omega$ be a bounded domain and $f:\Omega\rightarrow \mathbb{C}_2$. All the results in this subsection were taken out from the classical book on quaternionic analysis of Gürlebeck and Sprössig \cite{gurlebeck}

The Cauchy-Riemann operator has a right-inverse in the form
\begin{align}\label{T_operator}
    \left(Tf\right)(x)=-\frac{1}{\omega}\int_{\Omega}\frac{\overline{y-x}}{|y-x|^3}f(y)\, dy, \text{ for } x\in\Omega,
\end{align}
where $E(x,y)= -\frac{1}{\omega}\frac{\overline{y-x}}{|y-x|^3}$ is the generalized Cauchy kernel and $\omega=4\pi$ stands for the surface area of the unit sphere in $\mathbb{R}^3, $ that is, $DTf=f.$ This operator acts from $W^{k,p}(\Omega)$ to $W^{k+1, p}(\Omega)$ with $1<p<\infty$ and $k\in\mathbb{N}_0$.

Furthermore, we introduce the boundary integral operator
for $x\notin \partial\Omega$
\begin{align}\label{boundary_integral}
    \left(F_{\partial\Omega} f\right)(x) = \frac{1}{\omega}\int_{\partial\Omega} \frac{\overline{y-x}}{|y-x|^3}\alpha(y) f(y)\, dS(y),
\end{align}
where $\alpha(y)$ is the outward pointing normal unit vector to $\partial\Omega$ at $y$. We get the well-known Borel-Pompeiu formula
\begin{align*}
    \left(F_{\partial\Omega}f\right)(x) + \left(TDf\right)(x)=f(x) \text{ for } x\in\Omega.
\end{align*}
Obviously, $D F_{\partial\Omega}=0$ holds through this formula it it holds that $F_{\partial\Omega}$ acts from $W^{k-\frac{1}{p},p}(\partial\Omega)$ into $W^{k,p}(\Omega)$, for $k\in\mathbb{N}$ and $1<p<\infty$.

One of the other well-known results we will need for our work is the Plemelj-Sokhotzki formula is obtaining by taking the trace of the boundary integral operator. 

First we introduce an operator over the boundary of $\Omega$.
\begin{proposition}
    If $f\in W^{k, p}(\partial\Omega)$, then there exists the integral
    \begin{align}
\left(S_{\partial\Omega}f\right) = \frac{1}{2\pi} \int_{\partial\Omega} \frac{\overline{y-x}}{|y-x|^3}\alpha(y)f(y)\, dS(y)
    \end{align}
    for all points $x\in\Omega$ in the sense of Cauchy principal value.

    Furthermore, the operator $S_{\partial\Omega}$ is continuous in $W^{k,p}(\partial\Omega)$, for $1<p<\infty, \, k\in\mathbb{N}$.
\end{proposition}

From this the Plemelj-Sokhotzki formula is given as:
\begin{theorem}
    Let $f\in W^{k,p}(\partial\Omega)$ where by taking the non-tangential limit we have:
    \begin{align*}
        \lim_{\substack{x\rightarrow x_0,  \\ x\in\Omega,\, x_0\in\partial\Omega}} \left(F_{\partial\Omega}f\right)(x) = \frac{1}{2}\left(f(x_0) + \left(S_{\partial\Omega}f\right)(x_0)\right).
    \end{align*}
\end{theorem}

One of the corollaries concerns the limit to the boundary acting as a projector. That is,
\begin{corollary}
    The operator $P_{\partial\Omega}$ denoting the projection onto the space of all $\mathbb{H}-$valued functions which may be monogenicaly extended into the domain $\Omega$.\\

    Then, this projection may be represented as $$P_{\partial\Omega} = \frac{1}{2}\left(I+S_{\partial\Omega}\right).$$
\end{corollary}
The proofs of this results and others to follow in our proofs may be found in \cite{gurlebeck}.
\\

Now we are ready to start constructing our work on the inverse conductivity problem.

\section{Inverse Dirac scattering problem}

Transforming our conductivity equation into another type of equation also changes the inverse problem we are concerned. We transform it into a system of equations based on the Cauchy-Riemann operator $D$ (also called Dirac operator in some contexts) and thus we need to solve the inverse Dirac scattering problem first and only afterwards we care about the inverse conductivity problem.

Let $u$ be a solution to (\ref{Conductivity_eq1}) for some boundary function. We define
\begin{align*}
    \phi = \gamma^{1/2} \left(\,\bar{D}u, D u\right)^T,
\end{align*}
remark that $\gamma^{1/2}$ is well-defined since it is contained in $\mathbb{C}^+$. Then, $\phi$ solves the system
\begin{align}\label{Dirac_system}
    \begin{cases}
    D \phi_1 &= \phi_2q_1, \\
    \bar{D}\phi_2 &= \phi_1 q_2,
    \end{cases} \quad\text{ in } \mathbb{R}^3.
\end{align}
where $q_1 = -\frac{1}{2}\frac{\bar{D}\gamma}{\gamma}$ and $q_2 = -\frac{1}{2}\frac{D\gamma}{\gamma}$. 

This transformation arises as follows:

\begin{align*}
    D\phi_1 &= D\left(\gamma^{1/2}\bar{D}u\right) = D\gamma^{1/2}\bar{D}u + \gamma^{1/2}\Delta u  \\
    &= D\gamma^{1/2}\bar{D}u - \gamma^{-1/2}\nabla\gamma\cdot\nabla u \\
    &= D\gamma^{1/2}\bar{D}u - \frac{1}{2}\gamma^{-1/2}\left(D\gamma\bar{D}u+\D u\bar{D}\gamma\right) \\&=-\frac{1}{2}\left(\gamma^{1/2}D u\right)\frac{\bar{D}\gamma}{\gamma} = \phi_2 q_1
\end{align*}

Carefully, we can extend our potential to the outside by setting $\gamma\equiv 1$ outside of $\Omega$, which lead us to treat the study the equation in $\mathbb{R}^3$.

\subsection{Exponentially Growing Solutions}

We devise new exponentially growing solutions from the classical ones used in three dimensions. In most literature works, the exponential behavior is defined through the function $e^{x\cdot\zeta},$ with $\zeta\in\mathbb{C}^3$ fulfilling $\zeta\cdot\zeta=0$. However, in our scenario this function does not fulfill $D e^{ix\cdot\zeta}=0$, which brings the simplicity in all of the literature works.

Since we know that it is harmonic we can generate a monogenic function through it. Let $\zeta\in \mathbb{C}^3$ such that $\zeta\cdot\zeta:=\zeta_0^2+\zeta_1^2+\zeta_2^2 = 0$, then it holds
\begin{align*}
    \Delta e^{x\cdot\zeta} = 0 \Leftrightarrow D\left(\bar{D}e^{x\cdot\zeta}\right) = 0 \equiv D\left(e^{x\cdot\zeta}\,\bar{\zeta}\right)
\end{align*}

where now $\zeta$ is also defined as a quaternion through $\zeta=\zeta_0+e_1\zeta_1+e_2\zeta_2 \in \mathbb{C}_2.$ Thus the function $E(x,\zeta)=e^{x\cdot\zeta}\,\bar{\zeta}$ is monogenic. This also arises from the choice of $\zeta$, since $\zeta\,\bar{\zeta} = \zeta_0^2+\zeta_1^2+\zeta_2^2=0$.

We make a clear statement of when $\zeta$ is a complex-quaternion or complex-a vector, but in most cases it is clear from context: it is a vector if it is in the exponent and a quaternion otherwise.

We assume the following asymptotic behaviour for $\phi$:
\begin{align}\label{asymptotics}
    \phi_1 = e^{x\cdot\zeta}\bar{\zeta}\mu_1,\\
    \phi_2 = e^{x\cdot\bar{\zeta}^c}\bar{\zeta}^{c}\mu_2
\end{align}

Setting $\tilde{\mu}_1=\bar{\zeta}\mu_1$ and $\tilde{\mu}_2=\bar{\zeta}^{c}\mu_2$ we have the equations:
\begin{align}\label{Dirac_system_mu}
\begin{cases}
    D\tilde{\mu}_1 &= e^{-x\cdot\left(\zeta-\,\bar{\zeta}^{c}\right)}\tilde{\mu}_2q_1 \\ 
    \bar{D}\tilde{\mu}_2 &= e^{x\cdot\left(\zeta-\,\bar{\zeta}^{c}\right)}\tilde{\mu}_1q_2 \\ 
\end{cases}
\end{align}

Further, we assume $\tilde{\mu}\rightarrow \begin{pmatrix}
     1  \\
     0 
\end{pmatrix}$ as $|x|\rightarrow \infty.$ These system of equations will lead us to an integral equation from which we can extract interesting behaviour for $\zeta \rightarrow \infty$.

The main point of this subsection is to demonstrate how we can obtain the system of integral equations related with (\ref{Dirac_system_mu}). Here, the approach is similar to \cite{Lakshtanov_Tejero_Vainberg}, but we need to be careful due to the non-commutative nature of quaternions.

Recall, that $D T=\bar{D}\,\bar{T} = I$ (in appropriate spaces). Hence, applying this to (\ref{Dirac_system_mu}) it holds:
\begin{align*}
\begin{cases}
    \tilde{\mu}_1 = 1 + T\left[e^{-x\cdot\left(\zeta-\bar{\zeta}^{c}\right)}\tilde{\mu}_2q_1\right] \\
    \\
    \tilde{\mu}_2 = \overline{T}\left[e^{x\cdot\left(\zeta-\bar{\zeta}^{c}\right)}\tilde{\mu}_1q_2\right]
\end{cases}
\end{align*}

Thus, we can obtain two integral equations with respect to their function:

\begin{align*}
&\begin{cases}
    \tilde{\mu}_1 = 1 + T\left[e^{-x\cdot\left(\zeta-\bar{\zeta}^{c}\right)} \bar{T}\left[e^{x\cdot\left(\zeta-\bar{\zeta}^{c}\right)}\tilde{\mu}_1q_2\right] q_1\right] \\
    \\
    \tilde{\mu}_2 =   \bar{T}\left[e^{x\cdot\left(\zeta-\bar{\zeta}^{c}\right)} q_2\right] +  \bar{T}\left[e^{x\cdot\left(\zeta-\bar{\zeta}^{c}\right)}T\left[e^{-x\cdot\left(\zeta-\bar{\zeta}^{c}\right)}\tilde{\mu}_2 q_1\right]q_2\right] 
\end{cases}
\end{align*}
\begin{align}\label{Inversion_eq}
\begin{cases}
\tilde{\mu}_1 = 1 + M^1\tilde{\mu}_1 \\
\tilde{\mu}_2 = \overline{T}\left[e^{x\cdot\left(\zeta-\overline{\zeta}^{\mathbb{C}}\right)} q_2\right] + M^2\tilde{\mu}_2
\end{cases} \hspace{-0.4cm}\Leftrightarrow
\begin{cases}
[I - M^1](\tilde{\mu}_1-1) = M^1 1 \\
[I - M^2](\tilde{\mu}_2) = \bar{T}\left[e^{x\cdot\left(\zeta-\bar{\zeta}^{\mathbb{C}}\right)} q_2\right] 
\end{cases}
\end{align}

Our objective now is to study the uniqueness and existence of this equations, we approach this task by proving that $M^j,\, j=1,\,2$ are contractions. 

Instead of working with all possible $\zeta\in \mathbb{C}_{(2)}$ fulfilling $\zeta\bar{\zeta}=0$, we choose them for $k\in\mathbb{R}^3$ as
$$\zeta = k^{\perp} + i\frac{k}{2}, \quad k^{\perp}\cdot k=0$$
and $k^{\perp}$ can be algorithmically found.

We now describe our space of functions in terms of the space variable and $k\in\mathbb{R}^3$ as
\begin{align}
S=L^{\infty}_x(L^p_{k}(|k|>R))
\end{align}

where $R>0$ is a constant. In this space we prove that the operators $M^1,\, M^2$ are indeed contractions:
\begin{lemma}\label{contractions}
Let $p>2$. Then $$\lim_{R\rightarrow\infty} \|M^j\|_{S} = 0.$$
\end{lemma}

To further study the system (\ref{Inversion_eq}), we also need to show that the right-hand side is in $S$ for an $R$ large enough:
\begin{lemma}\label{right_hand_side}
Let $p>2$. Then there exists $R>0$ such that 
\begin{align}
    M^1 1 \in S,\\
    \label{T_in_S}\bar{T}\left[e^{x\cdot\left(\zeta-\bar{\zeta}^{\mathbb{C}}\right)} q_2\right] \in S
\end{align}
\end{lemma}

The above Lemmas imply the existence and uniqueness of $(\tilde{\mu}_1, \tilde{\mu}_2)$ solving the system (\ref{Inversion_eq}) with respect to the potential $q$. This is essential for the reconstruction procedure we show up next.

\subsection{Reconstruction from scattering data}

In this section, we are mixing ideas from \cite{Lakshtanov_Tejero_Vainberg} and \cite{Nachman3D} with quaternionic theory to obtain the potential from the scattering data.

Starting from Clifford-Green theorem
\begin{align*}
    \int_{\Omega} \left[g(x)\left(\bar{D}f(x)\right) + \left(g(x)\bar{D}\right)f(x)\right]\, dx = \int_{\Bound} g(x)\overline{\eta(x)}f(x)\, dS_x
\end{align*}
and using $g(x; i\xi+\zeta) = (i\xi+\zeta)e^{-x\cdot(i\xi+\zeta)}$ for $\xi\in\mathbb{R}^3$ such that $(i\xi+\zeta)\cdot(i\xi+\zeta)=0$. This implies that $g\bar{D}=0.$ Thus we define the scattering data as:
\begin{align}\label{Scattering_data}
    h(\xi,\zeta) = \left(i\xi+\zeta\right)\int_{\Bound} e^{-x\cdot(i\xi+\zeta)}\overline{\eta(x)}\phi_2(x,\zeta)\,dx
\end{align}
Applying now Clifford-Green theorem we obtain another form for the scattering data:
\begin{align*}
    h(\xi, \zeta) &= \left(i\xi+\zeta\right)\int_{\Omega} e^{-x\cdot(i\xi+\zeta)}\bar{D}\phi_2(x,\zeta)\,dx \\
    &= (i\xi+\zeta)\int_{\Omega} e^{-ix\cdot\xi}\left(e^{-x\cdot\zeta}\phi_1(x,\zeta)\right)q_2(x)\, dx, \quad \text{by } \overline{\D}\phi_2=\phi_1q_2\\
    &= (i\xi+\zeta)\int_{\Omega} e^{-ix\cdot\xi}\left(\zeta\mu_1(x,\zeta)\right)q_2(x)\, dx \\
    &= i\xi \int_{\Omega} e^{-ix\cdot\xi}\tilde{\mu}_1(x,\zeta)q_2(x)\, dx, \quad \textnormal{since } \bar{\zeta}\zeta =0 \\
    &= i\xi \hat{q}_2({\xi}) + i\xi\int_{\Omega} e^{-ix\cdot\xi}\left[\tilde{\mu}_1(x,\zeta)-1\right]q_2(x)\, dx.
\end{align*}

Thus, we have:
\begin{align}\label{reconstruction_1}
\hat{q}_2(\xi) = \frac{h(\xi,\zeta)}{i\xi} - \int_{\Omega} e^{-ix\cdot\xi} \left[\tilde{\mu}_1(x,\zeta)-1\right]q_2(x)\,dx
\end{align}
This is yet not enough to reconstruct the potential, since the integral acts as a residual in the reconstruction and requires data that we technically do not have. Therefore, we integrate everything over an annulus in $k$
\begin{align}\label{scattering_data_}
    \nonumber\int_{R<|k|<2R} \frac{\hat{q}_2(\xi)}{|k|^3}\, dk &= \frac{1}{i\xi}\int_{R<|k|<2R} \frac{h(\xi, \zeta(k))}{|k|^3}\, dk \\
    &- \int_{R<|k|<2R} \frac{1}{|k|^3}\int_{\Omega} e^{-ix\cdot\xi}\left[\tilde{\mu}_1(x,\zeta(k))-1\right]q_2(x)\, dx, 
\end{align}
since the potential does not depend on $k$ it can be taken out of the integral and taking the limit as $R\rightarrow \infty$ leads the second integral on the right to decay to zero, obtaining a reconstruction formula.
\begin{theorem}
Let $\Omega\subset\mathbb{R}^3$ a bounded Lipschitz domain, $q\in L^{\infty}(\Omega)$ be a complex-valued potential obtained through a conductivity $\gamma\in W^{1,\infty}(\Omega),\, \text{Re } \gamma\geq c>0.$ Then, we can reconstruct the potential from 
\begin{align}\label{potential_scattering}
    \hat{q}_2(\xi) = \lim_{R\rightarrow\infty} \frac{C}{i\xi} \int_{R<|k|<2R} \frac{h(\xi, \zeta(k))}{|k|^3}\, dk,
\end{align}
where $C = \frac{1}{4\pi\ln(2)}$.
\end{theorem}

\begin{proof}
    The scattering data is defined from the solutions of the Dirac system (\ref{Inversion_eq}) and therefore it holds that $\tilde{\mu}_1 - 1 \in S$. Starting from (\ref{scattering_data_}) we obtain by integrating the right-hand side for any $\xi\in\mathbb{R}^3$:

    \begin{align*}
        4\pi\ln 2\,\hat{q}_2(\xi) &= \frac{1}{i\xi}\int_{R<|k|<2R} \frac{h(\xi, \zeta(k))}{|k|^3}\, dk \\
        &- \int_{R<|k|<2R} \frac{1}{|k|^3}\int_{\Omega} e^{-ix\cdot\xi}\left[\tilde{\mu}_1(x,\zeta(k))-1\right]q_2(x)\, dx        
    \end{align*}

    Let $p>2$ and $1/p+1/q=1$. We estimate the last integral:
    \begin{align*}
        \Bigg|&\int_{R<|k|<2R} \frac{1}{|k|^3}\int_{\Omega} e^{-ix\cdot\xi}\left[\tilde{\mu}_1(x,\zeta(k))-1\right]q_2(x)\, dx\Bigg| \leq \\
        &\leq \int_{R<|k|<2R} \frac{1}{|k|^3}\int_{\Omega} \left|e^{-ix\cdot\xi}\left[\tilde{\mu}_1(x,\zeta(k))-1\right]q_2(x)\right|\, dx \\
        &\leq C_{\Omega}\|q\|_{\infty} \int_{R<|k|<2R} \frac{1}{|k|^3}\sup_x \left|\tilde{\mu}_1(x,\zeta(k))-1\right|\, dk \\
        &\leq C_{\Omega}\|q\|_{\infty} \left[\int_{R<|k|<2R}\frac{1}{|k|^{3q}}\, dk\right]^{1/q}\left[\int_{R<|k|<2R}  \sup_x \left|\tilde{\mu}_1(x,\zeta(k))-1\right|^p\, dk\right]^{1/p} \\
        &\leq C_{\Omega}\|q\|_{\infty} \|\tilde{\mu}_1-1\|_S \left[\int_{R<|k|<2R}\frac{1}{|k|^{3q}}\, dk\right]^{1/q}
    \end{align*}
    Taking the limit as $R\rightarrow 0$ the integral that is left goes to zero which implies the desired decay to zero and leaves us with our reconstruction formula.     
\end{proof}

Now, in order to connect the functions that solve the electrical conductivity equation (\ref{Conductivity_eq1}) and the solutions to the Dirac equation (\ref{Dirac_system}), which are exponential growing, we introduce the following result:

\begin{proposition}\label{poincare_lemma_type}
    Let $\Omega$ be a bounded domain in $\mathbb{R}^3$.
    Let $\phi=(\phi_1,\phi_2)$ be a solution of the Dirac system (\ref{Dirac_system}) for a potential $q\in L^{\infty}(\Omega)$ associated with the complex-conductivity $\gamma\in W^{1, \infty}(\Omega).$ 

    If $\phi_1=\bar{\phi}_2$ then there exists a unique solution $u$ of:
    \begin{align}
        \begin{cases}
        \bar{D}u = \gamma^{-1/2}\phi_1, \\
        D u = \gamma^{-1/2}\phi_2.
        \end{cases}
    \end{align}

    Further, this function fulfills the conductivity equation $$\nabla\cdot\left(\gamma\nabla u\right)=0 \text{ in } \Omega.$$
\end{proposition}

Let us recall the main theorem, that we are now able to prove with all these pieces we assembled.\\

\textbf{Theorem 1.1}
\textit{Let $\Omega\subset\mathbb{R}^3$ a bounded Lipschitz domain, $\gamma_i\in W^{1,\infty}(\Omega), \, i=1,\, 2$ be two complex-valued conductivities with $\textnormal{Re }\gamma_i\geq c>0.$ 
$$\text{If } \Lambda_{\gamma_1} = \Lambda_{\gamma_2}, \text{ then } \gamma_1=\gamma_2.$$}

\begin{proof}
Due to Theorem 3.3, one only needs to show that the scattering data $h$ for $|k|>>1$ is uniquely determined by the Dirichlet-to-Neumann map $\Lambda_{\gamma}$. Due to the lack of Poincaré Lemma in our current framework in quaternionic analysis with the $D$ and $\bar{D}$ operator, than a new technique is required to obtain a similar proof to \cite{Brown_Uhlmann}, for example.

For such, let us start with two conductivities $\gamma_1, \, \gamma_2$ in $W^{1,\infty}(\Omega)$ for $\Omega$ a bounded domain. By hypothesis $\Lambda_{\gamma_1}=\Lambda_{\gamma_2}$ and thus by \cite{Pombo_3D} we have $\left.\gamma_1\right|_{\partial\Omega}=\left.\gamma_2\right|_{\partial\Omega}$.

Further, we can extend $\gamma_j,\, j=1,2$ outside $\Omega$ in such a way that in $\mathbb{R}^3\setminus \Omega$ and $\gamma_j-1\in W^{1,\infty}_{\text{comp}}(\mathbb{R}^3)$. Let $q_j, \phi^j, \mu^j, h_j,\, j=1,2$ be the potential and the solution in (\ref{Dirac_system}), the function in (\ref{asymptotics}), and the scattering data in (\ref{Scattering_data}) all associated with the conductivity $\gamma_j$. 

Due to the scattering formulation at the boundary $\partial\Omega$, then we just want to know if $\phi^1=\phi^2$ on $\partial\Omega$ when $|k|>>1$.

First, by Proposition 3.4, we know that there exists an $u_1$ such that $$\phi^1=\gamma_1^{1/2}(\bar{D}u_1, Du_1)^T,$$ which is a solution to 
\begin{align*}
    \nabla\cdot(\gamma_1\nabla u_1)=0 \text{ in } \mathbb{R}^3.
\end{align*}

Now, let us define $u_2$ by
\begin{align*}
u_2 = \begin{cases}
    u_1 &\text{ in } \mathbb{R}^3\setminus \Omega,\\
    u &\text{ in } \Omega.
\end{cases}
\end{align*}

where $\hat{u}$ is the solution to the Dirichlet problem 
\begin{align*}
\begin{cases}
    \nabla\cdot\left(\gamma_2\nabla u\right)=0 &\text{ in } \Omega,\\
    u = u_1 &\text{ on } \partial\Omega.
\end{cases}
\end{align*}

Let $g\in C^{\infty}_c(\mathbb{R}^3).$ Then,
\begin{align*}
    \int_{\mathbb{R}^3} \gamma_2\nabla u_2\nabla g\, dx &= \int_{\mathbb{R}^3\setminus\Omega} \gamma_1\nabla u_1\nabla g\, dx + \int_{\Omega} \gamma_2\nabla \hat{u}\nabla g\, dx \\
    &= -\int_{\partial\Omega} \Lambda_{\gamma_1}\left[\left.u_1\right|_{\partial\Omega}\right]g\, ds_x + \int_{\partial\Omega} \Lambda_{\gamma_2}\left[\left.u\right|_{\partial\Omega}\right]g\, ds_x \\
    &= 0.
\end{align*}

Hence, $u_2$ is the solution of $\nabla\cdot\left(\gamma_2\nabla u_2\right)=0$ in $\mathbb{R}^3$. Further, the following function 
$$\psi^2=\gamma_2^{1/2}\left(\bar{D}u_2,Du_2\right)^T$$ is the solution of (\ref{Dirac_system}) where the potential is given by $\gamma_2$.

Furthermore, $\psi^2$ has the asymptotics of $\phi^1$ in $\mathbb{R}^3\setminus\Omega$, thus by Lemma 3.1 and 3.2 it will be the unique solution of the respective integral equation of (\ref{Dirac_system}). Thus, $\psi^2$ will be equal $\phi^2$ when $|k|>R$. Since, on the outside $\psi^2\equiv \phi^1$. Then we obtain:
\begin{align*}\phi^1=\phi^2 \text{ in } \mathbb{R}^3\setminus\Omega.\end{align*}

In particular, we have equality at the boundary $\partial\Omega$.
So, this implies that if the Dirichlet-to-Neumann maps are equal the respective scattering data will also be the same. Thus, the Dirichlet-to-Neumann map uniquely determines the potential $q$. 

From the definition of $q$, we can uniquely determine the conductivity $\gamma$ up to a constant, which in the end is defined by $\left.\gamma\right|_{\partial\Omega}$ which is uniquely determined by the Dirichlet-to-Neumann map $\Lambda_{\gamma}$.
\end{proof}

\section{Auxiliary Proofs}

\textbf{Proof of Lemma 3.1.}

Let us assume, without loss of generality, that $f$ is a scalar function. Further, we present the proof for $M^1$, since for $M^2$ it follows analogously.

Recall, that we choose $\zeta\in\mathbb{C}_{(2)}$ with respect to $k\in\mathbb{R}_{(2)}$ as
$$\zeta = k^{\perp} + i\frac{k}{2}, \quad k^{\perp}\cdot k=0.$$

In vector form, this leads to $\zeta-\zeta^c = ik$ which implies the following deductions:
\begin{align*}
    M^1f(x) &= \int_{\mathbb{R}^3} e^{-w\cdot\left(\zeta-\bar{\zeta}^c\right)}\frac{\overline{x-w}}{|x-w|^3}\int_{\mathbb{R}^3}e^{y\cdot\left(\zeta-\bar{\zeta}^c\right)} \frac{w-y}{|w-y|^3}f(y)q_2(y)\,dy\, q_1(w)\, dw \\ &= 
    \int_{\mathbb{R}^3} \int_{\mathbb{R}^3} e^{-i w\cdot k}\frac{\overline{x-w}}{|x-w|^3}e^{i y\cdot k} \frac{w-y}{|w-y|^3}f(y)q_2(y) q_1(w)\, dwdy \\
    &= \int_{\mathbb{R}^3} A(x, y; k) f(y)\, dy,
\end{align*}
where
\begin{align*}
    A(x, y; k) = \int_{\mathbb{R}^3} e^{-i(w-y)\cdot k}\frac{\overline{x-y}}{|x-y|^3}\frac{w-y}{|w-y|^3}q_2(y)q_1(w)\, dw.
\end{align*}
Due to the compact support of the potential $q_2$, it holds that $A$ has compact support on the second variable.

Let us now apply the norm in terms of $k$ to it:
\begin{align*}
    \|Mf(x,\cdot)\|_{L^p(|k|>R)} &= \left[\int_{|k|>R} |Mf(x,\zeta)|^p\, d\sigma_{\zeta}\right]^{1/p} \\ &= \left[\int_{|k|>R} \left|\int_{\Omega} A(x, y; k)f(y)\,dy\right|^p\, d\sigma_{k}\right]^{1/p} \\
    &\leq \int_{\Omega} \left[\int_{|k|>R} | A(x, y; k)f(y)|^p\, d\sigma_{k}\right]^{1/p}dy \\
    &\leq \int_{\Omega}\sup_{k} \left|A(x, y;k)\right|\, dy\, \|f\|_S.
\end{align*}
In order to complete the proof we show that the first integral goes to zero as $R\rightarrow\infty$.

Let $A^s$ be given with the extra factor $\alpha(s|x-w|)\alpha(s|w-y|)$ in the integrand, where $\alpha \in C^{\infty}$ is $1$ outside a neighborhood of the origin and $0$ inside a smaller neighborhood of it.

Since,
$$\int_{B_1(0)}\int_{B_1(0)} \frac{1}{|w|^2}\frac{1}{|w-y|^2}\, dw\,dy,$$
it holds that for any $\epsilon > 0$ there exists an $s>0$ such that:
\begin{align*}
    \int_{\Omega}|A-A^s|\, dy < \epsilon.
\end{align*}

Further, we denote $A^{s_0, n}$ the function $A^{s_0}$ with potentials $q_1, \,q_2$ replaced by their $L^1$ smooth approximation $Q_1^n,\, Q_2^n \in C^{\infty}$. Since the other factors are bounded it holds
$$\int_{\Omega} |A^{s_0} - A^{s_0, n}|\, dy < \epsilon.$$

Now it is enough to show that $A^{s_0, n}\rightarrow 0$ as $|k|\rightarrow 0$ uniformly!

All integrands inside of it will be in $C^{\infty}$ and uniformly bounded, thus by Riemann-Lebesgue the result follows.\hfill $\qedsymbol$
\\

\textbf{Proof of Lemma 3.2.}
Once again recall that $\zeta = \left(k^\perp + i\frac{k}{2}\right)$ for $k\in\mathbb{R}^3$.
First we show that $M^1 1\in S$.
We have
\begin{align*}
    M^1 1= \int_{\Omega}\int_{\Omega} e^{-iw\cdot k}\frac{\overline{x-w}}{|x-w|^3}\frac{w-y}{|w-y|^3}e^{iy\cdot k}q_2(y)q_1(w)\, dy\,dw, 
\end{align*}
and applying the $L^p$ norm in $k$ followed with Minkowski integral inequality we obtain
\begin{align*}
    \left[\int_{|k|>R} |M^1 1|^p dk\right]^{1/p} \leq \int_{\Omega} \frac{|q_1(w)|}{|x-w|^2}\left[\int_{|k|>R} \left|\int_{\Omega} e^{iy\cdot k} \frac{w-y}{|w-y|^3}q_2(y) dy\right|^pdk \right]^{1/p}dw
\end{align*}
The inner most integral resembles a Fourier transform, hence, applying the Hausdorff-Young inequality for $p>2$ we have
\begin{align*}
    \left[\int_{|k|>R} \left|\int_{\Omega} e^{iy\cdot k} \frac{w-y}{|w-y|^3}q_2(y)\, dy\right|^p\,dk \right]^{1/p} \leq \left[\int_{\Omega} \frac{|q_2(y)|^{p'}}{|w-y|^{2p'}}\, dy\right]^{1/p'} < C\|q_2\|_{\infty},
\end{align*}
where the last inequality follows quickly by Young's convolution inequality and Riesz type estimate of the kernel.

Therefore, by the same Riesz type estimate it holds:
\begin{align*}
    \left[\int_{|k|>R} |M^1 1|^p\, dk\right]^{1/p} \leq C\|q_2\|_{\infty} \int_{\Omega} \frac{|q_1(w)|}{|x-w|^2}\, dw \leq C' \|q_2\|_{\infty}\|q_1\|_{\infty}.
\end{align*}

To complete the proof we need to show statement (\ref{T_in_S}). Similarly, to the above proof, we have by Hausdorff-Young Inequality, Young's convolution inequality and a Riesz type estimate the following:
\begin{align*}
    \left[\int_{|k|>R} \left|\int_{\mathbb{R}^3} e^{i y\cdot k}\frac{x-y}{|x-y|^3}q_2(y)\, d\sigma_y \right|^p\, d\sigma_k \right]^{1/p} &\leq \left[\int_{\mathbb{R}^3} \left|\frac{x-y}{|x-y|^3}q_2(y)\right|^{p'}\, d\sigma_y \right]^{1/p'} \\ &\leq C\|q_2\|_{\infty}
\end{align*}
\hfill \qedsymbol
\\

We need the following auxiliary result for the proof of Proposition \ref{poincare_lemma_type}.
\begin{lemma}\label{average_lemma}
	Let $\Omega$ be a bounded Lipschitz domain in $\mathbb{R}^3$. 
	
	If $h$ is a scalar-valued and harmonic function that fulfills $$\textnormal{Vec}(S_{\partial\Omega}h)=0,$$ then $\left.h\right|_{\partial\Omega}$ is constant.
\end{lemma}

\begin{proof}
	First, note that $I+S_{\partial\Omega}=P_{\partial\Omega}$ is a projector and by Proposition 2.5.12 and Corollary 2.5.15 of \cite{gurlebeck} it holds that $P_{\partial\Omega}h$ is the boundary value of a monogenic function in $\Omega$.
	
	Since $h$ is a scalar-valued function it holds that
	\begin{align*}
		P_{\partial\Omega}h&= \text{Sc}(P_{\partial\Omega}h) + \text{Vec}(P_{\partial\Omega}h) \\
		&= (h+\text{Sc}_{\partial\Omega}h) + \text{Vec}(S_{\partial\Omega} h).
	\end{align*}

	Let $w=(h+\text{Sc}_{\partial\Omega}h)$ and $v=\text{Vec}(S_{\partial\Omega}h)$. Now, we denote $f$ as the monogenic extension of $P_{\partial\Omega}h$ in $\Omega$, as such, the boundary values of $f$ fulfill $\text{tr}f=w+v$. Note that by hypothesis we have that $\left.v\right|_{\partial\Omega}=0$. 
	
	Hence, $f$ is also an harmonic function, which implies that the scalar and vector components are harmonic. 
	\begin{align*}
		\begin{cases}
			\Delta (\text{Vec}f) = 0, \\
		\left.\text{Vec\,}f\right|_{\partial\Omega}=0.
		\end{cases}
	\end{align*}
	
	By a mean value theorem or a maximum principle it holds that $\text{Vec}f=0$. Due to this and $f$ being monogenic we obtain that $Df=0 \Leftrightarrow D(\text{Re}f)=0$. Thus, $\text{Re}f=c$ since $D$ is a quaternionic operator.
	
	Consequently, the boundary values are also constant, which means that $w=c$ in $\partial\Omega$. Since, $\text{Sc}(S_{\partial\Omega}h)$ is an averaging operator it holds that $h=c$.
\end{proof}

\textbf{Proof of Proposition 3.4}

Suppose that $(u, v)$ are solutions to the following equations:
\begin{align*}
    \begin{cases}
    \bar{D}u = \gamma^{-1/2}\phi_1 \\
    D v = \gamma^{-1/2}\phi_2.
    \end{cases}
\end{align*}

From applying the operator $D$ and $\bar{D}$ to the first and second equation respectively, we obtain from $\phi_2=\overline{\phi}_1^{\mathbb{H}}$ and $q_2 = \overline{q}_1^{\mathbb{H}}$ the following:

\begin{align*}
    \Delta u &= D(\gamma^{-1/2}\phi_1) = D(\gamma^{-1/2})\phi_1 + \gamma^{-1/2}D\phi_1 \\
    &= -\frac{1}{2}\gamma^{-3/2}(D\gamma\phi_1) + \gamma^{-1/2}\phi_2 q_1 \\
    &= \gamma^{-1/2}\left[q_2\phi_1 + \phi_2q_1\right] = \gamma^{-1/2}\left[\overline{q}_1^{\mathbb{H}}\phi_1 + \overline{\phi}_1^{\mathbb{H}}q_1\right]\\
    &= \gamma^{-1/2}\text{Sc } (\overline{\phi}_1^{\mathbb{H}}q_1).
\end{align*}
and 
\begin{align*}
    \Delta v &= \bar{D}(\gamma^{-1/2}\phi_2) = \bar{D}(\gamma^{-1/2})\phi_2 + \gamma^{-1/2}\bar{D}\phi_2 \\
    &= -\frac{1}{2}\gamma^{-3/2}(\bar{D}\gamma)\phi_2 + \gamma^{-1/2}\phi_1 q_2 \\
    &= \gamma^{-1/2}\left[q_1\phi_2 + \phi_1q_2\right] = \gamma^{-1/2}\left[q_1\overline{\phi}_1^{\mathbb{H}} + \phi_1\overline{q}_1^{\mathbb{H}}\right]\\
    &= \gamma^{-1/2}\text{Sc } (\phi_1\overline{q}_1^{\mathbb{H}}).
\end{align*}

The first thing to notice is that both equations imply that $u$ and  $v$ must be scalar-valued functions.

Further, notice that
\begin{align*}
    \Delta(u-v) &= \gamma^{-1/2}\left[\text{Sc } (\overline{\phi}_1^{\mathbb{H}}q_1) - \text{Sc } (\phi_1\overline{q}_1^{\mathbb{H}})\right]\\ &= \gamma^{-1/2}\left[\text{Sc } (\overline{\phi}_1^{\mathbb{H}}q_1) - \text{Sc } (\overline{q_1\overline{\phi}_1^{\mathbb{H}}})\right] = 0.
\end{align*}

Therefore, $h=u-v$ is an harmonic function. Our objective is to show that $h\equiv 0$, thus showing that $u=v$.

For such, let us consider the theory of integral transforms in quaternionic analysis. We have
\begin{align*}
    u &= \bar{T}(\gamma^{-1/2}\phi_1) + \overline{F}_{\partial\Omega}(\gamma^{-1/2}\phi_1) \text{ and } \\
    u &=  \bar{T}(\gamma^{-1/2}\phi_1) + \overline{F}_{\partial\Omega}(u),
\end{align*}

which implies that $$\overline{F}_{\partial\Omega}(\gamma^{-1/2}\phi_1) = \overline{F}_{\partial\Omega}u.$$

Analogously, we obtain $$F_{\partial\Omega} (\gamma^{-1/2} \phi_2) = F_{\partial\Omega} v.$$

Here, we can extrapolate from the first equation and from $u$ being scalar-valued that \begin{align*}
    \overline{\gamma^{-1/2}\phi_1}F_{\partial\Omega} = F_{\partial\Omega} u \\
    \Leftrightarrow \gamma^{-1/2}\phi_2 F_{\partial\Omega} = F_{\partial\Omega} u.
\end{align*}

Applying the operator $F_{\partial\Omega}$ on the other side, we obtain:
\begin{align*}
    F_{\partial\Omega}^2u &= F_{\partial\Omega}(\gamma^{-1/2}\phi_2)F_{\partial\Omega} \text{ and }\\
    F_{\partial\Omega}^2v &= F_{\partial\Omega}(\gamma^{-1/2}\phi_2)F_{\partial\Omega} \\
    &\Rightarrow F_{\partial\Omega}^2 h = F_{\partial\Omega}^2 (u-v) = 0.
\end{align*}

If we take the trace on both sides, the operator becomes a projector thus we obtain $\text{tr }F_{\partial\Omega}h=0$.

Now, through the Sokhotski-Plemelj formula we obtain:
\begin{align*}
    \text{tr } F_{\partial\Omega} h = \left.h\right|_{\partial\Omega} + S_{\partial\Omega}h = 0, \text{ at } \partial\Omega.
\end{align*}

Since $h$ is a scalar-valued function that we decompose this formulation with the scalar and vector part to obtain two conditions:
\begin{align*}
    \begin{cases}
    &h + \text{Sc} (S_{\partial\Omega} h) = 0 \\
    &\text{Vec}(S_{\partial\Omega} h) = 0.
    \end{cases}
\end{align*}

Through the second condition and Lemma \ref{average_lemma} we have that $h$ is constant over $\partial\Omega$.

Now, given that $h$ is a scalar constant, the first condition reduces to:
\begin{align*}
    h(1+\text{Sc}(S_{\partial\Omega} 1))=0
\end{align*}

By \cite{gurlebeck} we obtain that $1+\text{Sc}(S_{\partial\Omega}1) = 1/2$ in $\partial\Omega$. Therefore, we conclude that $h\equiv 0$ in $\partial\Omega$. Given that $h$ is harmonic, this immediately implies that $h=0$ in $\Omega$.

Therefore, we obtain $u=v$, and therefore there exists a unique solution to the initial system through the $T$ and $F_{\partial\Omega}$ operators in $\Omega$.

To finalize, we only need to show that $u$ fulfills the conductivity equation in $\Omega$.

Bringing the first equation to light
$$\bar{D}u=\gamma^{-1/2}\phi_1,$$
changing the side of the conductivity we get $\gamma^{1/2}\bar{D}u = \phi_1$ and applying the $D$ operator to both sides now brings
\begin{align*}
    \quad &D\left(\gamma^{1/2}\bar{D}u\right)=D\phi_1 \\
    \Leftrightarrow \quad &D\left(\gamma^{1/2}\right)\bar{D}u + \gamma^{1/2}\Delta u = \phi_2 q_1 \\
    \Leftrightarrow \quad
    &D\left(\gamma^{1/2}\right)\bar{D}u + \gamma^{1/2}\Delta u = \gamma^{-1/2}Du \frac{1}{2}\frac{\bar{D}\gamma}{\gamma} \\
    \Leftrightarrow \quad &\frac{1}{2}\gamma^{1/2}D\gamma\bar{D}u + \gamma^{1/2}\Delta u + \frac{1}{2}Du\frac{\bar{D}\gamma}{\gamma^{1/2}}=0 \\
    \Leftrightarrow\quad &\nabla\gamma\cdot\nabla u + \gamma\Delta u = 0 \Leftrightarrow \nabla\cdot\left(\gamma\nabla u\right) = 0
\end{align*}
\hfill\qedsymbol\\

As such, we conclude our proof of uniqueness for complex-conductivities in $W^{1,\infty}(\Omega)$ from the Dirichlet-to-Neumann map $\Lambda_{\gamma}$. Notice that (\ref{potential_scattering}) even provides a reconstruction formula, but as mentioned in the previous section it is very unstable for computational purposes.

\end{document}